\documentclass{amsart}
\usepackage{amsmath, amsthm, amscd, amsfonts, amssymb}
\usepackage{hyperref}
\usepackage{mathrsfs}
\usepackage{xcolor}
\usepackage{tikz}
\usetikzlibrary{matrix,arrows}
\newtheorem{thm}{Theorem}[section]
\newtheorem{lem}[thm]{Lemma}
\newtheorem{pro}[thm]{Proposition}
\newtheorem{cor}[thm]{Corollary}
\theoremstyle{definition}
\newtheorem{den}[thm]{Definition}
\newtheorem{example}[thm]{Example}
\theoremstyle{remark}

\numberwithin{equation}{section}

\begin{document}
\title[  module tensorizing maps  ] {  module tensorizing  maps on $ C^*$-algebras }
\author[ A. Shirinkalam ]{Ahmad Shirinkalam}



\address{Department of Mathematics,  Central Tehran Branch, Islamic Azad University, Tehran, Iran}
\email{a.shirinkalam@iau.ac.ir, a.shirinkalam-science@iaustb.ac.ir}

\begin{abstract}
For the $ C^*$-algebras  $ \mathfrak{A}, A$ and $ B $ where $ A $ and $ B $ are $  \mathfrak{A} $-bimodules with  compatible actions, we consider  amalgamated	$  \mathfrak{A} $-module tensor product of $ A $ and $ B $ and  study its relation  with the  C*-tensor product of $A$ and $B$ for the min and max norms.	We introduce and study the notions of module tensorizing maps, module exactness, and module nuclear pairs of $ C^*$-algebras in this setting. We give concrete examples of $C^*$-algebras on inverse semigroups. 
\end{abstract}

\maketitle

\section{Introduction and Preliminaries}
Finding  suitable $ C^*$-norms to complete the algebraic tensor product of  $ C^*$-algebras goes back to the works of  T. Turumaru \cite{Tu} in 1952  and Alain Guichardet \cite{G} in 1969. Since then a lot of work has been done on the subject, including the Takesaki result confirming  that the spatial tensor norm is the minimal $C^*$-norm on the  algebraic tensor product \cite{Ta}. 

There are some important notions related to the tensor product of  $ C^*$-algebras like nuclearity and exactness,   studied  among others  by Alain Connes (1978), Uffe Haagerup (1983), Simon Wassermann (1994) and Eberhard Kirchberg. The equivalence between tensor product conditions and approximation properties is now known due to fundamental results of Choi and Effros (1978), Kirchberg (1977) and Pisier (1995). 

The nuclearity (equality of min and max $C^*$-norms on algebraic tensor product with any other $C^*$-algebra) is known to be equivalent to amenability for $C^*$-algebras. B. E. Johnson in   \cite{J} showed that a locally compact group $ G $ is amenable  if and only if the group algebra $ L^1(G) $ is amenable as a Banach algebra.
A version of Johnson's result for (inverse) semigroups is proved by Massoud Amini, using  the notion of module amenability introduced in \cite{A} for Banach algebras (see also \cite{SPA}). The relation between  module nuclearity and module injectivity for $ C^*$-algebras were settled in \cite{AR}.

In this paper we further explore the notion of module tensor products of $ C^*$-algebras (specially in the minimal and  maximal tensor norms) and following Pisier \cite{P}, study tensorizing maps in the module context. This has applications in the module version of the notion of exactness, studied in the classical case by Kirchberg and Wassermann for group $C^*$-algebras \cite{kw} and for general $C^*$-algebras by Kirchberg \cite{k}.

Throughout the paper $ \mathfrak{A}, A$ and $ B$ denote  $ C^*$-algebras where $ A $ and $ B $ are $  \mathfrak{A} $-bimodules with  compatible actions. Also $ \odot $ denotes the algebraic tensor product of the algebras $ A $ and $ B $. The completion of $ A\odot B $ with respect to the minimal (maximal) tensor norm is denoted by  $ A\otimes_{\rm min} B $ and  $ A\otimes_{\rm max} B $, respectively.

    We say that $  \mathfrak{A}$ acts on  $ A $ (with a compatible action) if there are two maps $\mathfrak{A}\times A\rightarrow A;(\alpha , a)\mapsto\alpha\cdot a  $ and $ A\times \mathfrak{A}\rightarrow A;(a,\alpha)\mapsto a\cdot \alpha $ satisfying the following conditions:

\vspace{.2cm}
(i) both of the maps are linear on each variable,

(ii)  there is $ k>0 $ such that 
$ \| a\cdot\alpha \|, \| \alpha \cdot a \| \leq k \| a \|_{A} \| \alpha \|_{\mathfrak{A}} $,

(iii) $ \alpha \cdot (ab)=(\alpha \cdot a)b $ and $(ab)\cdot\alpha =a(b\cdot\alpha) $,

(iv) $ (\alpha \cdot a)^{*}=a^*\cdot \alpha^* $ and $ (a \cdot \alpha)^*=\alpha^* \cdot a^*, $

\noindent for each $ \alpha \in \mathfrak{A} $ and $ a, b \in  A$.
 \vspace{.2cm}
  
  The natural example of such action is the case where $ A $ acts on itself by the algebra multiplication.
 If $ A$  is an $  \mathfrak{A}$-bimodule with  the compatible action, then so is the dual space $  A^* $.

 % % % % % % % % % % % % % % % % % % % % % % % % % % % % % % % % % % % % % % % % % %
  Let $ X $ be a Banach $ A$-bimodule and a Banach $  \mathfrak{A}$-bimodule such that for every $\alpha \in   \mathfrak{A}, a \in A  $ and $ x \in X, $
 \begin{eqnarray*}\label{compatible} 
 	\alpha \cdot (a\cdot x)=(\alpha\cdot a)\cdot x, \quad a\cdot (\alpha \cdot x)=(a\cdot \alpha)\cdot x, \quad (\alpha \cdot x)\cdot a = \alpha\cdot (x\cdot a), 
 \end{eqnarray*}
 and the same for the right action. Then we say that $ X $ is a Banach
 $ A $-$  \mathfrak{A}$-bimodule. If moreover, for each $ \alpha\in   \mathfrak{A}$ and $ x \in X $, $ \alpha\cdot x =x\cdot \alpha $, then $ X $ is called a commutative Banach
 $ A $-$  \mathfrak{A}$-bimodule. In this case $ X^* $, the dual of $ X $ is also a commutative Banach
 $ A $-$  \mathfrak{A}$-bimodule with the canonical action. Let $ Y  $ be another Banach $ A $-$  \mathfrak{A}$-bimodule, then an $ A $-$  \mathfrak{A}$-bimodule morphism from $ X $ to $ Y $ is a bounded linear map $ \varphi :X \rightarrow Y $ which is module morphism with respect to both actions, that is,
 $$\varphi (a \cdot x)=a\cdot \varphi (x), \quad \varphi (x\cdot a)=\varphi(x)\cdot a, \quad  \varphi (\alpha \cdot x)=\alpha\cdot \varphi (x), \quad \varphi (x\cdot \alpha)=\varphi(x)\cdot \alpha,  $$ for each $\alpha \in   \mathfrak{A}, a \in A$ and $ x \in X.  $
% % % % % % % % % % % % % % % % % % % % % % % % % % % % % % % % % % % % % % % % % % % % % % % % % %

	When $ A $ acts on itself by the algebra multiplication, it is not in general a Banach $A $-$  \mathfrak{A}$-bimodule,
	since we have not assumed the condition
	\begin{eqnarray}\label{com} 
		a  (\alpha \cdot a^\prime)= (a\cdot \alpha)a^\prime \quad(\alpha \in \mathfrak{A}, a,a^\prime \in A).
	\end{eqnarray}
	
	However, if  $A $ is a commutative  $  \mathfrak{A}$-bimodule, then the equation (\ref{com}) holds, so  $ A $ is a Banach  $A$-$  \mathfrak{A}$-bimodule.

Let $ A,B $ be $  \mathfrak{A}$-bimodules.
Let $ I_{A,B} $ be the  ideal of
$ A\odot B $ generated by the elements of the form
$ a\cdot\alpha \otimes b - a\otimes \alpha\cdot b, $ for  $ \alpha \in \mathfrak{A}, a \in A $ and $ b\in B $. We use $ I $ instead of $  I_{A,B} $ when there is no ambiguity.
The quotient space $ \dfrac{ A\odot B}{I} $ is denoted by $ A \odot_{\mathfrak{A}}B $. It is desirable to relate the $C^*$-completions of $ A \odot_{\mathfrak{A}}B $ with quotients of $C^*$-completions of $ A \odot B. $ This is done in the next section.

Let $ \pi:A \odot A \rightarrow A;\sum a_i\otimes a^\prime_i\mapsto \sum a_ia^\prime_i $ be the multiplication map  and let $ J $  be the closed ideal of $ A $ generated by $ \pi (I_{A,A})$, that is, the closed linear span of  the set of elements of the form $ (a\cdot\alpha)  a^\prime - a( \alpha\cdot a^\prime ) \:(\alpha \in \mathfrak{A}, a,a^\prime \in A) $.

\section{Module tensor products and module tensorizing maps}

In this section we explore module versions of tensorizing maps. The following definition is due to G. Pisier \cite[Definition 7.1]{P}.
\begin{den}\label{tens}
	Let $ A $ and $ B $ be  $ C^*$-algebras and let $ \gamma $  and $ \eta $ be one of the symbols min  or max.
	A linear map $ u:A\rightarrow B  $ is called $(\gamma \rightarrow \eta)$-tensorizing if for every $ C^* $-algebra $ C $,  the map $  u\odot{\rm id}_C:A\odot C\rightarrow B\odot C $ is $ \gamma$-$\eta$-continuous, equivalently for every $ x\in A\odot C  $ we have 
	$$\| u\odot {\rm id}_C(x) \|_{\eta}\leq \| x \|_{\gamma}.  $$
\end{den}

In the next definition we extend the notion of  tensor product of $ C^*$-algebras to  module tensor product.
\begin{den}\label{main}
Let $ \gamma $ be a $ C^* $-norm on $ A\odot B $. If the quotient norm on  $ A \odot_{\mathfrak{A}}B $ coincides with the norm $ \gamma $, then we define the $ \mathfrak{A} $-module $ \gamma $-tensor product of $ A $ and $ B $ as the completion of $ A \odot_{\mathfrak{A}}B $ with respect to this norm and  denote it  by $ A\otimes_{\mathfrak{A},\gamma}B $. In this case for $ x\in A\odot B $ we have $ \|x+I\|_\gamma:=\inf_{i\in I}\| x+i\|_\gamma $.
\end{den}	

The next result now follows immediately.

\begin{pro}	\label{ideal}
Let $ \gamma $ be a $ C^* $-norm on $ A\odot B $ and 	
let $ I_\gamma $ be the closure of $ I $ in this norm. Then we have an isometric isomorphism 	
	$$ (A\otimes_{\gamma}B)/ I_\gamma \simeq A\otimes_{\mathfrak{A},\gamma}B, $$
where 	$  A\otimes_{\gamma}B $ is the completion of $ A\odot B $ with respect to $ \gamma $.
\end{pro}

In particular, it follows that the module max-norm is well behaved with respect to inclusion of ideals.
	
\begin{cor}	\label{2.3}
Let $ N $ be a closed two-sided	ideal of a $ C^* $-algebra $ A $. Then for any 
$ C^* $-algebra $ B $ we have an isometric embedding $N\otimes_{\mathfrak{A},\rm max}B\subseteq A\otimes_{\mathfrak{A},\rm max}B  $.
\end{cor}	
\begin{proof}
	By 	\cite[Lemma 7.14]{P}, we have $ N\otimes_{\rm max}B\subseteq  A\otimes_{\rm max}B  $.
Let $ \bar I_{N,B} $ be the closed ideal in $ N\otimes_{\rm max}B $ generated by the elements of the form 
$ n\cdot \alpha\otimes b-n\otimes \alpha \cdot b $. Then $ I_N\subseteq I $. Hence 
$ N\otimes_{\mathfrak{A},\rm max}B\subseteq   A\otimes_{\mathfrak{A},\rm max}B $.
\end{proof}

More generally, let $ D $ be a $ C^* $-subalgebra of a 
$ C^* $-algebra $ A $. The inclusion $ D\subseteq A $ is called  max-injective \cite[7.2]{P}
whenever for every $ C^* $-algebra $ C $ the inclusion $ D\otimes_{\rm max} C\subseteq A\otimes_{\rm max} C $ holds. We extend this definition to the case of module tensor product of $ C^* $-algebras.

\begin{den}
	The inclusion $ D\subseteq A $ is called  $ \mathfrak{A} $-max-injective if for every $ \mathfrak{A} $-bimodule $ C $ with compatible actions, the maximal norm on $ D\odot_\mathfrak{A} C $ coincides with the norm inherited from $ A\otimes_{\mathfrak{A}, \rm max} C $. Equivalently, the map $  D\otimes_{\mathfrak{A},\rm max} C\rightarrow  A\otimes_{\mathfrak{A},\rm max} C $ is injective (and so isometric).
\end{den}
As an example, if $ D $ be a closed two-sided ideal of $ A $, then by the above corollary, $ D\subseteq A $ is $ \mathfrak{A} $-max-injective.

The case of min-norm is less problematic. If  $ D\subseteq A $ the restriction of the min-norm of $ A\otimes_{\mathfrak{A},\rm  min} B $  to $ D\odot_\mathfrak{A}B $  always coincides  with the norm of $ D\otimes_{\mathfrak{A},\rm  min} B $. This means that if  $ D\subseteq A $, then for every $ B $, we have $ D\otimes_{\mathfrak{A},\rm min}B\subseteq A\otimes_{\mathfrak{A},\rm min}B. $

Next let us treat the quotient norms. Let $ u:A\rightarrow B  $ be a  completely bounded (completely positive) $\mathfrak{A}  $-linear map. The following facts are well-known:

\vspace{.2cm}
(i) The map $ u\odot{\rm id}_C:A\odot C\rightarrow B\odot C $ induces a  map $ u\odot_{\mathfrak{A}} {\rm id}_C:A\odot_{\mathfrak{A}} C\rightarrow B\odot_{\mathfrak{A}} C.  $ 

(ii) The  map $ \tilde{u}:A/J_A \rightarrow B/J_B; a+J_A\mapsto u(a)+J_B $ is  completely bounded (completely positive)  with $ \| \tilde{u} \|_{c.b.}\leq \| u\|_{c.b.} $.
\vspace{.2cm}

\begin{den}\label{a-max}
	For every $ x\in A\odot B $ we define 
	$$ \|x+I \|_{\mathfrak{A}{\rm- max}}=\sup\| \pi(x+I) \|$$
	where the supremum runs over all the *-homomorphisms of 
	$A\odot_{\mathfrak{A}} B  $. Similarly, for each $ x=\sum_{i} a_i\otimes b_i\in A\odot B $ we define 
	$$ \|x+I \|_{\mathfrak{A}{\rm-min}}= \|\textstyle\sum_i \pi(a_i)\otimes \sigma (b_i) \|_{B(\mathcal{H}\bar{\otimes}\mathcal{K})}.  $$
\end{den}

The next result guarantees that the definitions \ref{main} and \ref{a-max} are compatible for the module min and max norms.

\begin{thm}\label{main1}
	For $\gamma=\rm min$ or $\rm max $, and each $ x\in A\odot B $ we have
	\begin{equation*}
		\|x+I \|_{\mathfrak{A}-{\gamma}}= \|x+I \|_{\gamma}.
	\end{equation*}
\end{thm}
\begin{proof}
	First let us treat the case of $\gamma$=max: 
	Let $ \rho:\mathfrak{A}\rightarrow B(\mathcal{H}) $ be a *-homomorphism. Using this, one could define a left and a right  action of $ \mathfrak{A} $ on $ B(\mathcal{H}) $ via $ \alpha \cdot T:=\rho(\alpha)T $ and $ T\cdot \alpha:=T\rho(\alpha), $ for $ \alpha \in \mathfrak{A} $ and $ T\in B(\mathcal{H}) $. This makes $ B(\mathcal{H}) $  an $ \mathfrak{A} $-bimodule with compatible actions.
	Let $ \pi:A\rightarrow B(\mathcal{H}) $ and $ \sigma:B\rightarrow B(\mathcal{H}) $ be $ \mathfrak{A} $-bimodule *-representations with commuting ranges. Then there is a unique $ \mathfrak{A} $-bimodule *-homomorphism $ \pi\times \sigma:A\odot B\rightarrow  B(\mathcal{H});a\otimes b\mapsto \pi(a)\sigma(b)$. Let $ I:=I_{A,B} $. If $\pi\times \sigma\arrowvert_{I}=0  $, then we obtain  a unique  *-representation $ \pi\times_{\mathfrak{A}} \sigma:A\odot_{\mathfrak{A}} B\rightarrow  B(\mathcal{H})$. 
	
	Conversely, for every  *-homomorphism $ \varphi:A\odot_{\mathfrak{A}} B\rightarrow  B(\mathcal{H}) $, the *-homomorphism $ \psi:A\odot B\rightarrow  B(\mathcal{H}) $ defined by $ \psi(a\otimes b)=\varphi(a\otimes b+I) $  vanishes on $ I $. Extending $\psi$ to the algebraic tensor product of unitizations of $A$ and $B$ if needed, we may assume that $A$ and $B$ are unital. Putting $ \pi(a):=\psi(a\otimes 1_B) $ and $ \sigma(b):=\psi(1_A\otimes b), $ we obtain $ \mathfrak{A} $-bimodule *-representations
	$ \pi:A\rightarrow B(\mathcal{H}) $ and $ \sigma:B\rightarrow B(\mathcal{H}) $ with commuting ranges. In other words, there is a one-one correspondence between the pairs $ (\pi,\sigma) $ of *-representations on $ A $ and $ B $ with commuting ranges with $\pi\times \sigma\arrowvert_{I}=0  $ and the *-homomorphisms of 
	$A\odot_{\mathfrak{A}} B  $.
	
	By the above definition, the norm $ \mathfrak{A}$-$\rm max $ has the universal property. Indeed, if $ \varphi:A\odot_{\mathfrak{A}} B\rightarrow  B(\mathcal{H}) $ is a *-homomorphism, then there is a unique *-homomorphism $ \tilde{\varphi}:A\otimes_{\mathfrak{A},{\rm max}} B\rightarrow  B(\mathcal{H}) $
	that extends $ \varphi $. In particular, for every $ \mathfrak{A} $-bimodule *-representations $ \pi:A\rightarrow B(\mathcal{H}) $ and $ \sigma:B\rightarrow B(\mathcal{H}) $ with commuting ranges and $\pi\times \sigma\arrowvert_{I}=0  $, there is a unique *-representations $ \pi\times_{\mathfrak{A}}\sigma:A\otimes_{\mathfrak{A},{\rm max}} B\rightarrow B(\mathcal{H})$.
	In particular, for every $ x\in A\odot B$, 
	
	\begin{equation*}
		\|x+I \|_{\mathfrak{A}{\rm -max}}=\inf_{i\in I} \|x+i \|_{\rm max}. 
	\end{equation*} 
	
	Next we treat the case of $\gamma$=min: 
	Let $ \pi:A\rightarrow  B(\mathcal{H})$ and $ \sigma:B\rightarrow  B(\mathcal{K}) $ be injective *-representations. Then $$ \pi\odot\sigma:A\odot B\rightarrow  B(\mathcal{H})\odot  B(\mathcal{K})\subseteq  B(\mathcal{H}\bar{\otimes}\mathcal{K});\ \textstyle\sum_i a_i\otimes b_i\mapsto\textstyle\sum_i \pi(a_i)\otimes \sigma (b_i)$$ is  an injective *-homomorphism \cite[Proposition 3.1.12]{BO}. Again by passing to minimal unitizations we may assume that $A$ and $B$ are unital. 
	Let $I=I_{A,B}$. Then $I$ is the closure of the span $I_0$ of elements of the form $a\cdot\alpha\otimes b-a\otimes\alpha\cdot b$. Define $\mathfrak{A}$-actions on $B(\mathcal{H})$ and $B(\mathcal{K})$ as follows:
	$$x\cdot\alpha:=x\pi(1_A\cdot \alpha), \ \ \alpha\cdot y:=\sigma(\alpha\cdot 1_B)y,$$
	for $\alpha\in \mathfrak{A}, x\in B(\mathcal{H})$ and $y\in B(\mathcal{K})$. Let $I^{'}
	$ be the corresponding ideal of $\pi(A)\odot \sigma(B)$. Then clearly $(\pi\odot\sigma)(I)\subseteq I^{'} $, thus there is a *-homomorphism 
	$$ \pi\odot_{\mathfrak{A}}\sigma:A\odot_{\mathfrak{A}} B\rightarrow  \pi(A)\odot_{\mathfrak{A}} \sigma(B);\ \textstyle\sum_i(a_i\otimes b_i+I)\mapsto\textstyle\sum_i(\pi(a_i)\otimes \sigma (b_i))+I^{'},$$
	which is also a module map. Let us observe that $ \pi\odot_{\mathfrak{A}}\sigma$ is injective: if $\pi\odot_{\mathfrak{A}}\sigma(x+I)=0$, for some $x\in A\odot B$, then $(\pi\odot\sigma)(x)\in I^{'}$ and so it is the norm limit of a sequence of elements of the form $(\pi\odot\sigma)(y_i)$, with $y_i\in I_0$. Since $\pi\odot\sigma$ is injective, it follows that the sequence $(y_i)$ is Cauchy and so convergent to some $y\in I$. Thus we have $(\pi\odot\sigma)(x)=(\pi\odot\sigma)(y)$, and again by  injectivity of $\pi\odot\sigma$, $x=y\in I$, that is $x+I=0$.
	Now since $\pi(A)\odot_{\mathfrak{A}} \sigma(B)$ is a pre-C*-algebra, it has a faithful representation $\rho: \pi(A)\odot_{\mathfrak{A}} \sigma(B)\to B(\mathcal L)$ in a Hilbert space $\mathcal L$, and 
	$$\rho\circ (\pi\odot_{\mathfrak{A}}\sigma):A\odot_{\mathfrak{A}} B\to  B(\mathcal L)$$
	is an injective *-homomorphism.

	Note that by a similar argument to that of \cite[Proposition 3.3.11]{BO}, the  $ \mathfrak{A}$-$\rm min $ norm is independent of the choice of faithful representations $ \pi $ and $ \sigma $.

	Now  let $ \pi:A\rightarrow  B(\mathcal{H})$ and $ \sigma:B\rightarrow  B(\mathcal{K}) $ be injective *-representations. Then
	\begin{equation}\label{2.7}
		 \|x+I \|_{\mathfrak{A}{\rm-min}}= \|\rho\circ (\pi\odot_{\mathfrak{A}}\sigma)(x+I)\|=\inf_{i^{'}\in I^{'}}\| \pi\odot \sigma(x)+i^{'}\|=\inf_{i\in I}\| x+i\|=\|x+I \|_{\rm min},
	\end{equation}
	  
	for each  $ x\in A\odot B. $
\end{proof}

The next definition could be compared with Definition \ref{tens}.
\begin{den}\label{key}
	Let $ \gamma $  and $ \eta $ be one of the symbols min  or  max.
	An $\mathfrak{A}  $-linear map $ u:A\rightarrow B  $ is called $\mathfrak{A}$-$(\gamma \to \eta)$-tensorizing if for every $ C^* $-algebra $C$ that is an $\mathfrak{A}$-bimodule with compatible actions, we have 
	$$\| u\odot_{\mathfrak{A}}{\rm id}_C(\bar{x}) \|_{\eta}\leq \| \bar{x} \|_{\gamma},  $$
	for every $ \bar{x}\in A\odot_{\mathfrak{A}} C,  $  that is, the map $  u\odot_{\mathfrak{A}} {\rm id}_C:A\odot_{\mathfrak{A}} C\rightarrow B\odot_{\mathfrak{A}} C $ is $ \gamma$-$\eta$-continuous. 
	
	In this case the map $  u\odot_{\mathfrak{A}}{\rm id}_C $ extends to a contraction $u\otimes_{\mathfrak{A}}{\rm id}_C:A\otimes_{\mathfrak{A},\gamma} C\rightarrow B\otimes_{\mathfrak{A},\eta} C.  $
\end{den}

\begin{lem}\label{123}
If $ u :A\rightarrow B $ is an $\mathfrak{A}  $-module map, then for each $ C^*$-algebra $ C $ and  $ x\in A\odot C, $ 
\begin{equation}
	\| u\odot_{\mathfrak{A}} {\rm id}_C(\bar{x})\|_{\rm min}= \| u\odot {\rm id}_C(x)\|_{\rm min}
\end{equation}
where $ \bar{x}\in A\odot_{\mathfrak{A} } C $ is the image of $ x $ in the canonical embedding.
\end{lem}
\begin{proof}
Suppose that $ \pi: B\rightarrow B(\mathcal{H}) $ and $ \sigma:C \rightarrow B(\mathcal{K}) $ are injective *-representations and let $ x=\sum a_i\otimes c_i $. We denote $ I_{B,C} $ by $ I$.  By  (\ref{2.7})
\begin{align*}
\|  u\odot_{\mathfrak{A}} {\rm id}_C(\bar{x})\|_{\rm min}&=\| \sum u(a_i)\otimes c_i+I\|_{\rm min}\\&=\| \sum \pi u(a_i)\otimes \sigma (c_i)\|_{B(\mathcal{H}\bar{\otimes}\mathcal{K})}\\&=\| u\odot {\rm id}_C(x)\|_{\rm min}, 
\end{align*} 
where the last equality is just the definition of the spatial tensor product.
\end{proof}
\begin{cor}
	If $ u :A\rightarrow B $ is  $\mathfrak{A}$-$(\rm min \to \rm min)$-tensorizing, then it is $(\rm min \to \rm min)$-tensorizing.
\end{cor}

Amini and Rezavand in  \cite{AR} defined the notion of $ \mathfrak{A} $-nuclearity for a $ C^*$-algebra $ A $ in the  case where the action of $ \mathfrak{A} $ on $ A $ is left trivial. Here we extend this definition to an arbitrary action.
\begin{den}
		Let $ A $ and $ B $ be  $ C^*$-algebras with compatible actions of $ \mathfrak{A} $ on them. The pair $ (A,B) $ is called $ \mathfrak{A} $-nuclear if $  A\otimes_{\mathfrak{A},{\rm min}} B\simeq A\otimes_{\mathfrak{A}, {\rm max}} B$, or equivalently, the min and the max norm on $ A\odot_{\mathfrak{A}} B $ defined by Definition \ref{main}. coincide. A $ C^*$-algebra $ A $
		is called $ \mathfrak{A} $-nuclear if for every $ C^*$-algebra $ B $  with  a compatible   $ \mathfrak{A}$-action, the pair $ (A,B) $ is  $ \mathfrak{A} $-nuclear.
\end{den}

The  next result  immediately follows from definition.
\begin{pro}
A $ C^*$-algebra $ A $
is  $ \mathfrak{A} $-nuclear if and only if the identity map $ {\rm id}:A\rightarrow A $ is $\mathfrak{A}$-$({\rm min}\to{\rm max})$-tensorizing.
\end{pro}

If $ D\subseteq A $ is   $ \mathfrak{A} $-max-injective and the pair $ (A,B) $ is $ \mathfrak{A} $-nuclear, then so is $ (D,B) $. This is because, by  $ \mathfrak{A} $-nuclearity of  $ (A,B) $,
the restriction of the min-norm and the max-norm on $ D\odot_{\mathfrak{A}}B $ are the same, that is,
$$ D\otimes_{\mathfrak{A},\rm max}B=\overline{ D\odot_{\mathfrak{A}}B}^{\rm max}= \overline{ D\odot_{\mathfrak{A}}B}^{\rm min}=D\otimes_{\mathfrak{A},\rm min}B\subseteq A\otimes_{\mathfrak{A},\rm min}B. $$
Note that the first equality holds as   $ D\subseteq A $ is  $ \mathfrak{A} $-max-injective.

% % % % % % % % % % % % % % % % % % % % % % % % % % % % % % % % % % % % % % % % % % % % % % % %
% %     trivial action
\section{Comparing tensors for  trivial one-sided actions}

	We say that the  action of $ \mathfrak{A} $ on $ A $ is left-trivial if $ \alpha \cdot a=f(\alpha)a, $ for some character $ f\in \mathfrak{A}^* $. The right-trivial action of $ \mathfrak{A} $ on $ A $ is defined similarly.

Let $ \mathfrak{A} $ act on $ A $ left-trivially.
Let $ J_0 $ be the closed linear span of elements of the form $ a\cdot \alpha -f(\alpha)a $.

\begin{thm}\label{triv}
	Let $ \mathfrak{A}, A $ and $ B $ be unital.
	
	(i)  If the  actions of $\mathfrak{A}  $ on $ A $ and $ B $ are left-trivial with the same character $  f\in \mathfrak{A}^* $, then we have the isomorphic isomorphism 
	\begin{equation}
		A\otimes_{\mathfrak{A},{\rm max}}B\simeq \dfrac{A}{J_A}\otimes_{\rm max}B.
	\end{equation}
	If in addition  the $ C^* $-algebra $ B $ is exact, then 
	\begin{equation}
		A\otimes_{\mathfrak{A},{\rm min}}B\simeq \dfrac{A}{J_A}\otimes_{\rm min}B.
	\end{equation}

	(ii)  If the  actions of $\mathfrak{A}  $ on $ A $ and $ B $ are right-trivial with the same character $  f\in \mathfrak{A}^* $, then  we have the isomorphic isomorphism
	\begin{equation}
		A\otimes_{\mathfrak{A},\rm max}B\simeq A\otimes_{\rm max}\frac{B}{J_B}.
	\end{equation} 
	In addition if the $ C^* $-algebra $ A $ is exact, then
	\begin{equation}
		A\otimes_{\mathfrak{A},\rm min}B\simeq  A\otimes_{\rm min}\frac{B}{J_B}.
	\end{equation} 
\end{thm}
\begin{proof}
	(i) Let $(e_i)\subseteq  A $ be a bounded approximate identity. Then $$ a\cdot \alpha -f(\alpha)a=\lim_i[(a \cdot \alpha)e_i -a(\alpha \cdot e_i)]\in J_A, $$ thus $ J_0\subseteq J_A $.
	Conversely, for every $ a_1, a_2 \in A $ since $ J_0 $ is an ideal, we have
	$$ (a_1\cdot \alpha)a_2-a_1(\alpha \cdot a_2)=(a_1\cdot \alpha-f(\alpha)a_1)a_2\in J_0 ,$$ and we get $ J_A\subseteq J_0.$ Thus $ J_A= J_0 $. By \cite[Proposition 3.7.1]{BO}, exactness of the sequence 
	$$ 0\rightarrow J_A \rightarrow A\rightarrow A/J_A\rightarrow 0 $$ implies exactness of $$ 0\rightarrow J_A\otimes_{\rm max}B \rightarrow A\otimes_{\rm max}B\rightarrow A/J_A\otimes_{\rm max}B\rightarrow 0$$ hence we have an isomorphic isomorphism 
	$$ \dfrac{A\otimes_{\rm max}B}{J_A\otimes_{\rm max}B}\simeq \dfrac{A}{J_A}\otimes_{\rm max}B. $$
	Now it is enough to show that $ J_A\otimes_{\rm max}B=I_{\rm max} $. Since the action of $ \mathfrak{A} $ on $ B $ is left-trivial, $ I_{\rm max}  $ is spanned by the elements of the form 
	$ a\cdot \alpha -f(\alpha)a\otimes b $. This means that $ J_0\odot B $ spans $ I_{\rm max}  $. Thus 
	$ J_0\subseteq J_A $ implies $ I_{\rm max}\subseteq J_A\otimes_{\rm max}B.  $ Conversely $ J_A\subseteq J_0 $ gives $ J_A\odot B\subseteq I_{\rm max} $. Thus by \cite[Proposition 7.19 (iii)]{P}, $ J_A\otimes_{\rm max}B= \overline{J_A\odot B}^{\rm max}\subseteq I_{\rm max} $. Therefore, $  I_{\rm max}= J_A\otimes_{\rm max}B $ and $ A\otimes_{\mathfrak{A},\rm max}B\simeq \dfrac{A}{J_A}\otimes_{\rm max}B. $
	
	If $ B $ is exact,  exactness of the sequence 
	$$ 0\rightarrow J_A \rightarrow A\rightarrow A/J_A\rightarrow 0 $$ implies the exactness of $$ 0\rightarrow J_A\otimes_{\rm min}B \rightarrow A\otimes_{\rm min}B\rightarrow A/J_A\otimes_{\rm min}B\rightarrow 0, $$ thus 
	$$ \dfrac{A\otimes_{\rm min}B}{J_A\otimes_{\rm min}B}\simeq \frac{A}{J_A}\otimes_{\rm min}B. $$
	A similar argument as above shows that $I_{\rm min}= J_A\otimes_{\rm min}B $ and $ A\otimes_{\mathfrak{A},\rm min}B\simeq \dfrac{A}{J_A}\otimes_{\rm min}B. $
	
	(ii) The proof is similar as the above.
\end{proof}

\begin{cor}\label{3.3}
	Let $ u:A\rightarrow B  $ be an $\mathfrak{A}  $-module map. Suppose that the  actions of $\mathfrak{A}  $ on $ A $ and $ B $ are left-trivial with the same character $  f\in \mathfrak{A}^* $. Then the following statements are equivalent.

(i)  The map $ u:A\rightarrow B  $ is 	$\mathfrak{A}$-{\rm (max}$ \rightarrow${\rm max)}-tensorizing,

(ii)  The map 	 $ \tilde{u}:A/J_A \rightarrow B/J_B$ is {\rm (max}$\rightarrow${\rm max)}-tensorizing.
\end{cor}

\begin{cor}\label{3.4}
		 	Let $ u:A\rightarrow B  $ be an $\mathfrak{A}  $-linear map. Suppose that the  actions of $\mathfrak{A}  $ on $ A $ and $ B $ are left-trivial with the same character $  f\in \mathfrak{A}^* $.  Then 
		 	
(i)	If   $ \| u \| \leq 1 $ and $u$ is c.p., then 	$ u $ is 	$\mathfrak{A}$-$({\rm max}\rightarrow{\rm max})$-tensorizing.	 	
		 	
(ii) If $\| u \|_{c.b.}\leq1 $, then $ u $ is $\mathfrak{A}$-$(\rm min \to \rm min)$-tensorizing. 	
\end{cor}
\begin{proof}

(i) As observed above, we have $ \| \tilde{u}:A/J_A \rightarrow B/J_B \|\leq 1 $. 
Using \cite[Corollary 7.8]{P} $ \tilde{u} $ is $(\rm max \rightarrow \rm max)$-tensorizing. Therefore, by Corollary \ref{3.3}
$ u $ is 	$\mathfrak{A}$-{\rm (max$\rightarrow$max)}-tensorizing.

(ii)  
 By \cite[Corollary 7.2]{P} $\tilde{u}  $ is $(\rm min \to \rm min)$-tensorizing. Thus for each 
$ C^* $-algebra $ C $,
$$ \|\tilde{u} \odot {\rm id}_C :\frac{A}{J_A}\odot C \rightarrow \frac{B}{J_B}\odot C\| \leq1.  $$ 
Let $ I_A=I_{A,C} $ and $ I_B=I_{B,C} $. By a similar argument as in the proof of Theorem \ref{triv}, we have $ \dfrac{A}{J_A}\odot C \simeq \dfrac{A \odot C}{I_A} $ and  $ \dfrac{B}{J_B}\odot C \simeq \dfrac{B \odot C}{I_B} $. Thus
$$ \|u \odot_{\mathfrak{A}} {\rm id}_C :\dfrac{A \odot C}{I_A} \rightarrow\dfrac{B \odot C}{I_B}\| \leq 1,  $$
which means that $ u $ is $\mathfrak{A}$-$(\rm min \to \rm min)$-tensorizing. 
\end{proof}

\begin{pro}
	Suppose that the  action of $\mathfrak{A}  $ on $ A $ is left-trivial. Then
	the inclusion $ D\subseteq A $ is   $ \mathfrak{A} $-$\rm max$-injective in each of the following cases:
	
	(i)	 there is a c.p. projection $ \mathfrak{A} $-module map $ P:A\rightarrow D $ with $ \| P\|\leq1 $,
	
	(ii)	 there is a net of  c.p.  $ \mathfrak{A} $-module maps $ P_i:A\rightarrow D $ with $ \| P_i\|\leq1 $ and $ P_i\rightarrow {\rm id}_D$ in point-norm.
\end{pro}
\begin{proof}
	The result is immediate, since the maps $ P $ and $ P_i $'s are $\mathfrak{A}$-$(\rm max \rightarrow \rm max)$-tensorizing by Corollary \ref{3.4}.
\end{proof}

\begin{pro}
 Let $ N $ be a closed two-sided ideal of $ A $. If the  actions of $\mathfrak{A}  $ on $ A $ and $ B $ are right-trivial with the same character $  f $, then we have   isometric isomorphism
 $$ \frac{A}{N}\otimes_{\mathfrak{A},\rm max}B\simeq (A\otimes_{\mathfrak{A},\rm max}B)/(N\otimes_{\mathfrak{A},\rm max}B). $$
\end{pro}
\begin{proof}
	Since $ N $ is an ideal by \cite[Proposition 7.15]{P}, the sequence 
	$$0\rightarrow N\otimes_{\rm max}\frac{B}{J_B}\rightarrow A \otimes_{\rm max}\frac{B}{J_B}\rightarrow\frac{A}{N}\otimes_{\rm max}\frac{B}{J_B}\rightarrow 0$$
	is exact. Now by assumption  and 
 Theorem \ref{triv} (ii), we have $ 	A\otimes_{\mathfrak{A},\rm max}B\simeq A\otimes_{\rm max}\frac{B}{J_B} $. Hence the sequence 
$$0\rightarrow N\otimes_{\mathfrak{A},\rm max}B\rightarrow A\otimes_{\mathfrak{A},\rm max}B\rightarrow  \frac{A}{N}\otimes_{\mathfrak{A},\rm max}B\rightarrow 0 $$
is exact. Thus	
	$ \frac{A}{N}\otimes_{\mathfrak{A},\rm max}B\simeq (A\otimes_{\mathfrak{A},\rm max}B)/(N\otimes_{\mathfrak{A},\rm max}B), $ as required.
\end{proof}

\section{Module exactness}
Let $ A $ be a $ C^* $-algebra and $ N $ be a closed two-sided  ideal of $ A $. Then for every 
$ C^* $-algebra $ B $ the sequence
$$0\rightarrow N\odot B \rightarrow A\odot B\rightarrow \frac{A}{N} \odot B \rightarrow 0 $$
is exact. The key point here is that the completion of this sequence in the spatial tensor norm fails to be exact in general. A $ C^* $-algebra $ B $ is called exact if for every $ C^* $-algebra $ A $ and two-sided closed ideal $ N $, the sequence 
$$0\rightarrow N\otimes_{\rm min} B \rightarrow A\otimes_{\rm min} B\rightarrow \frac{A}{N}\otimes_{\rm min}  B\rightarrow 0 $$
is an exact sequence.

Here we extend this definition and introduce the notion of module exactness for a $ C^* $-algebra with a compatible one-sided $ \mathfrak{A} $-action.
\begin{den}
 Let  $B $ be a $ C^* $-algebra with a compatible left $ \mathfrak{A} $-action. Then
 $ B $ is called $ \mathfrak{A} $-exact if for every $ C^* $-algebra $ A $ with a compatible right $ \mathfrak{A} $-action and every closed two-sided  ideal and right submodule $ N $ of $ A $, the sequence
 $$0\rightarrow N\otimes_{\mathfrak{A},\rm min} B \rightarrow A\otimes_{\mathfrak{A}, \rm min} B\rightarrow \frac{A}{N}\otimes_{\mathfrak{A},\rm min} B \rightarrow 0 $$
 is exact.
\end{den}

In the above situation, whenever we need to treat $A$ and $B$ as two sided $\mathfrak{A}$-modules, we fix a character $f$ on $\mathfrak{A}$ and let the right action on $B$ and the left action on $A$ be trivial actions given by $f$. 

\begin{thm} \label{exact}
A  $ C^* $-algebra $ B $ with a compatible left $ \mathfrak{A} $-action is $ \mathfrak{A} $-exact if and only if $ \frac{ B}{J_B} $ is exact.
\end{thm}
\begin{proof}
We need to verify that the exactness of the first and the third rows in the following diagram are equivalent:

\begin{center}
		\begin{tikzpicture}
			\matrix [matrix of math nodes,row sep=1cm,column sep=1cm,minimum width=1cm]
			{
				|(A)| \displaystyle 0  &   |(B)| N\otimes_{\mathfrak{A},\rm min} B & |(C)|  A\otimes_{\mathfrak{A}, \rm min} B  &  |(D)|  \frac{A}{N}\otimes_{\mathfrak{A},\rm min} B  &  |(E)|  0     \\
				|(F)|    0     &   |(G)|    N\otimes_{\rm min} B & |(H)| A\otimes_{ \rm min} B &  |(I)|  \frac{A}{N}\otimes_{\rm min} B & |(J)| 0 \\
				|(K)|   0    &   |(L)|  N\otimes_{\rm min} \frac{B}{J_B}  & |(M)| A\otimes_{ \rm min} \frac{B}{J_B}  &  |(N)| \frac{A}{N}\otimes_{\rm min} \frac{B}{J_B} & |(O)|  0 \\
						};
		\draw[->]    (A)-- node [above] {}(B);
		\draw[->]    (B)-- node [above] {}(C);
		\draw[->]    (C)-- node [above] {}(D);
		\draw[->]    (D)-- node [above] {}(E);
		\draw[->]    (F)-- node [above] {}(G);
		\draw[->]    (G)-- node [above] {}(H);
		\draw[->]    (H)-- node [above] {}(I);
		\draw[->]    (I)-- node [above] {}(J);
		\draw[->]    (K)-- node [above] {}(L);
		\draw[->]    (L)-- node [above] {}(M);
		\draw[->]    (M)-- node [above] {}(N);
		\draw[->]    (N)-- node [above] {}(O);
		\draw[->]    (B)-- node [left] {$ q_N^\mathfrak{A} $}(G);
		\draw[->]    (C)-- node [left] {$ q_A^\mathfrak{A} $}(H);
		\draw[->]    (D)-- node [left] {$ q_{A/N}^\mathfrak{A} $}(I);
		\draw[->]    (G)-- node [left] {${\rm id}_N\otimes_{\rm min} q_B $}(L);
		\draw[->]    (H)-- node [left] {${\rm id}_A\otimes_{\rm min} q_B $}(M);
		\draw[->]    (I)-- node [left] {${\rm id}_{\frac{A}{N}}\otimes_{\rm min} q_B $}(N);
		\end{tikzpicture}
	\end{center}
	
First let us assume that $\frac{B}{J_B}$ is exact. Then by definition we have
$$	 \dfrac{A\otimes_{\rm min}\frac{B}{J_B}}{N\otimes_{\rm min}\frac{B}{J_B}}\simeq \dfrac{A}{N}\otimes_{\rm min}\frac{B}{J_B}. $$
Let the right action on $B$ and the left action on $A$ be trivial actions given by the same character $f$ on $\mathfrak{A}$. Let $ I_{A,B} $ be the  ideal of
$ A\odot  \frac{B}{J_B}$ generated by the elements of the form
$ a\cdot\alpha \otimes \bar b - a\otimes \alpha\cdot \bar b, $ for  $ \alpha \in \mathfrak{A}, a \in A $, $\bar b\in  \frac{B}{J_B}$, and let $ \bar I_{A, \frac{B}{J_B}} $ be its min-closure. For $\tilde B:=B/J_B$ we have $J_{\tilde B}=0$, and  since the right $\mathfrak{A}$-action on $B$ is assumed to be trivial, both the left and right $\mathfrak{A}$-actions on $\tilde B=B/J_B$ are trivial (given by the above character $f$). Thus by Theorem \ref{triv}, we have  $ \bar I_{A, \frac{B}{J_B}}=J_A\otimes_{\rm min}   \frac{B}{J_B}$. Therefore, we have
$$A\otimes_{\mathfrak{A}, {\rm min}}\frac{B}{J_B}\cong \dfrac{A\otimes_{\rm min}\frac{B}{J_B}}{\bar I_{A, \frac{B}{J_B}}}
\cong \dfrac{A\otimes_{\rm min}\frac{B}{J_B}}{J_A\otimes_{\rm min}   \frac{B}{J_B}}
\cong \frac{A}{J_A}\otimes_{\rm min}\frac{B}{J_B}.$$
Similarly, using  closed ideals $ \bar I_{N,B} $ and $ \bar I_{\frac{A}{N},B}, $ we get  
$$N\otimes_{\mathfrak{A}, {\rm min}}\frac{B}{J_B}\cong  \frac{N}{J_N}\otimes_{\rm min}\frac{B}{J_B}, \ \ \frac{A}{N}\otimes_{\mathfrak{A}, {\rm min}}\frac{B}{J_B}\cong  \frac{A/N}{J_{A/N}}\otimes_{\rm min}\frac{B}{J_B}. $$ 

Now let $q:=q_N: A\to A/N$ be the quotient map. This induces a map 
$$\tilde q: \frac{A}{J_A}\to  \frac{A/N}{J_{A/N}};\ \ x+J_A\mapsto q(x)+ J_{A/N}.$$
Let us observe that $\tilde q$ is well-defined: If $x\in J_A$ then $x$ could be approximated in $A$ with a linear combination of elements of the form $y=(a\cdot \alpha)a^{'}-a(\alpha\cdot a^{'})$ with $q(y)=(q(a)\cdot \alpha)q(a^{'})-q(a)(\alpha\cdot q(a^{'})\in J_{A/N}$. Thus $q(x)\in J_{A/N}$. 

Next let us consider the embedding
$$\iota: \frac{N}{J_N}\hookrightarrow  \frac{A}{J_{A}};\ \ n+J_N\mapsto n+ J_{A}.$$
To see that $\iota$, which is clearly well defined, is one-one, we need to observe that $J_N=J_A\cap N$: given a generator $x=(n\cdot \alpha)n^{'}-n(\alpha\cdot n^{'})$ of $J_N$, we have $x\in N$ (as $N$ is an ideal) and $x\in J_A$ (as $N\subseteq A$), thus  $J_N\subseteq J_A\cap N$. Conversely, given $\varepsilon>0$, each 
$x\in J_A\cap N$ is $\varepsilon$-close in $A$ to a linear combination of elements of the form $y=(a\cdot \alpha)a^{'}-a(\alpha\cdot a^{'})$. Given a bounded approximate identity $(e_i)$ of $N$ with $0\leq e_i\leq 1$ in the minimal unitization of $N$, each element  $e_ixe_i$ is  $\varepsilon$-close in $A$ to a linear combination of elements  $e_iye_i=(e_ia\cdot \alpha)a^{'}e_i-e_ia(\alpha\cdot a^{'}e_i)\in J_N$  (as $N$ is an ideal in $A$), thus  $x=\lim_i e_ixe_i\in J_N$ (as $J_N$ is close), that is, $J_A\cap N\subseteq J_N.$ 

Identifying $\frac{N}{J_N}$ with its image in $ \frac{A}{J_{A}}$, we claim that $\ker(\tilde q)=\frac{N}{J_N}$: given $\bar x=x+J_A$ with $\tilde q(\bar x)=0$, we have $q(x)\in J_{A/N}$, thus, given $\varepsilon>0$,  
$q(x)$ is $\varepsilon$-close in $A/N$ to a linear combination $z$ of elements of the form $(q(a)\cdot \alpha)q(a^{'})-q(a)(\alpha\cdot q(a^{'}))$, and so there is $n\in N$ such that $x-n$ is $2\varepsilon$-close in $A$ to a linear combination $y$ of elements of the form $(a\cdot \alpha)a^{'}-a(\alpha\cdot a^{'})$. Hence $\bar x$ is $2\varepsilon$-close in $A/J_A$ to $n+J_A=\iota(n+J_N)$, that is, $\bar x\in N/J_N$, after the above identification. Conversely, $$\tilde q(\iota(n+J_N))= \tilde q(n+J_A)=q(n)+J_{A/N}=J_{A/N},$$ 
and the formula for the kernel is proved. This means that 
$$\frac{A/J_A}{N/J_N}\cong  \frac{A/N}{J_{A/N}},$$
canonically. Finally we have
$$\dfrac{A\otimes_{\mathfrak{A}, {\rm min}}B}{N\otimes_{\mathfrak{A}, {\rm min}}B}\cong \dfrac{A/J_A\otimes_{\rm min}\frac{B}{J_B}}{N/J_N\otimes_{\rm min}\frac{B}{J_B}}\cong \dfrac{A/J}{N/J_N}\otimes_{\rm min}\frac{B}{J_B} \cong\dfrac{A/N}{J_{A/N}}\otimes_{\rm min}\frac{B}{J_B}
\cong \dfrac{A}{N}\otimes_{\mathfrak{A}, {\rm min}}B.,$$
that is, $B$ is $\mathfrak{A}$-exact. The converse is proved by going over the above argument in a reverse direction. 
\end{proof}

To provide a concrete example, let us discuss $C^*$-algebras of inverse semigroups. Following Duncan and Paterson \cite{DP}, let us consider  an inverse semigroup $ S $ and left regular representation $ \lambda: S\rightarrow \mathcal{B}(\ell^2(S))$. Then the completion of the image of $\ell^1(S)   $ in the representation induced by $  \lambda $ is called the reduced 
$ C^* $-algebra of $ S $, denoted by $ C^*_r(S) $. Let the full $ C^* $-algebra of $ S $  be denoted by $ C^*(S) $.

Let $ E $ be be  the set of idempotents of $ S $. Let $ \ell^1(E) $ act on $ \ell^1(S)  $ by multiplication from right and trivially from left, that is,
$$ \delta_s\cdot \delta_e= \delta_s\ast \delta_e=\delta_{se}\quad  \hbox{and} \quad \delta_e
   \cdot\delta_s=\delta_s \qquad(e\in E, s\in S). $$
These actions continuously extend to actions of $ C^*(E) $ on $ C^*(S) $ and $ C_r^*(S) $,  where the left actions are trivial \cite{AR}. Let $ \approx $ be an equivalence relation on $ S $ given by
$$ s\approx t \Leftrightarrow \delta_s-\delta_t\in J_0.  $$
Here the  ideal $ J_0 $  is the  linear span of the set $\{ \delta_{set}-\delta_{st}:s,t \in S, e\in E \}\subseteq \ell^1(S).  $ Since $ E $ is a semilattice, the discussion before \cite[Theorem 2.4]{ABE} 
shows that $ G_S:=S/\approx $ is a discrete group, that is known to be  the maximal group homomorphic image of $ S $.

Following the group case, we  define exact inverse semigroups as follows.
\begin{den}
An inverse semigroup $ S $ is called exact if  $ C_r^*(G_S) $ is an exact $ C^* $-algebra.
\end{den}
\begin{example}
Let $ S $ be an inverse semigroup with the set of idempotents $ E $. Suppose that $ E $ acts on $ S $ by multiplication from right and trivially from left as above.
Then the following are equivalent;

(i) $ S $ is an exact semigroup.

(ii) $ C_r^*(S) $ is $ C^*(E) $-exact.

\end{example}
\begin{proof}
Let $ J $  be the closure of $ J_0$ in $ C_r^*(S) $. The map $  \ell^1(S)\rightarrow   C_r^*(G_S); \delta_s\mapsto \delta_{[s]} $ is a continuous *-homomorphism with dense range $ \ell^1(G_S) $ so it lifts to a surjective *-homomorphism $ \varphi:   C_r^*(S)\rightarrow  C_r^*(G_S)$ with $ \ker (\varphi)=J $. Therefore, $ C_r^*(G_S)\simeq \dfrac{ C_r^*(S)}{J} $.
Now the result follows from Theorem \ref{exact}. 
\end{proof}

\section{decomposable module maps}

Let $A$ be an operator system and $B$ a
C*-algebra, both having $\mathfrak A$-bimodule structures as above. We will denote by $D_{\mathfrak A}(A,B)$ the set of all $\mathfrak A$-{\it decomposable} maps
$u : A \to B$, that is, the module maps that are in the linear span of $CP_{\mathfrak A}(A,B)$ of c.p. $\mathfrak A$-module maps. Following  Haagerup \cite{h} we  define 
$$\|u\|_{{\mathfrak A}-{\rm dec}} := \inf\max\{\|S1\|,\|S2\|\},$$ where the infimum runs over all module maps $S_1,S_2 \in CP_{\mathfrak{A}}(A,B)$ such that 
$$V:=\begin{pmatrix}
S_1& u\\
u_*& S_2
\end{pmatrix}
\in CP(A, \mathbb M_2(B)),
$$
where $u_*(a)=u(a^*)^*$. Note that $V$ is automatically a module map in the canonical module structure of $\mathbb M_2(B)$. By an argument similar to \cite[Lemma 6.2]{P},  $D_{\mathfrak A}(A,B)$ is a Banach space. 
We may ignore the module structure and compare this norm to $\|.\|_{{\rm dec}}$. Since in the latter case the infimum is taken over all c.p. maps $S_1$ and $S_2$ (not just the module c.p. maps), it is clear that $\|u\|_{{\mathfrak A}-{\rm dec}}\geq\|u\|_{{\rm dec}}$. In particular, 
\begin{equation}\label{6.1}
 \|u\|_{{\mathfrak A}-{\rm dec}}\geq\|u\|_{{\rm cb}}=\|u\|,
\end{equation}
for a c.p. map $u$, by \cite[Lemma 6.5(i)]{P}.  On the other hand, when $u$ is c.p., $\begin{pmatrix}
	S_1& u\\
	u_*& S_2
\end{pmatrix}
\in CP(A, \mathbb M_2(B)),$ thus 

\vspace{.2cm}
$(i)$ if $u$ is c.p., then
$$\|u\|_{{\mathfrak A}-{\rm dec}} = \|u\|_{\rm cb}=\|u\|. $$

\vspace{.2cm} 
Similarly, as in \cite[Lemma 6.5(ii)-(iii)]{P}, we have
 
\vspace{.2cm}
$(ii)$ if $u = u_*$ then
 $$\|u\|_{{\mathfrak A}-{\rm dec}} = \inf\{\|u_1 + u_2\|: u_1,u_2 \in CP_{\mathfrak A}(A,B), u = u_1 - u_2\}. $$
 
\vspace{.2cm} 
$(iii)$ if $\bar u= \begin{pmatrix}
	0& u\\
	u_*& 0
\end{pmatrix}$, then  $u \in D_{\mathfrak A}(A,B)$ if and only if $\bar u \in D_{\mathfrak A}(A,\mathbb M_2(B))$ and  $\|u\|_{{\mathfrak A}-{\rm dec}}=\|\bar u\|_{{\mathfrak A}-{\rm dec}}$. 
 
\vspace{.2cm}
Also, by the same argument as in  \cite[Lemma 6.6]{P},
\vspace{.2cm}

$(iv) D_{\mathfrak A}(A,B) \subseteq CB_{\mathfrak A}(A,B)$ and
$\|u\|_{{\mathfrak A}-{\rm dec}}\geq \|u\|_{{\rm cb}}$, for a c.b. module map  $u$, 
  \vspace{.2cm}
  
$(v)$ if  $u \in D_{\mathfrak A}(A,B)$ and $v \in D_{\mathfrak A}(B,C)$ then $vu \in D_{\mathfrak A}(A,C)$ and
$$\|vu\|_{{\mathfrak A}-{\rm dec}}\leq\|u\|_{{\mathfrak A}-{\rm dec}}\|v\|_{{\mathfrak A}-{\rm dec}}.$$

Now consider the situation that $u:A\to M$ is a module map from a $C^*$-algebra to a von Neumann algebra, both with $\mathfrak A$-bimodule structures. For $M$, we  the left and right module actions are assumed to be $w^*$-continuous. This makes $M_*\leq M^*$ an $\mathfrak{A}$-submodule. Now since $u^*:M^*\to A^*$ is also a module map in the canonical module structure of the dual spaces, so is its restriction to $M_*$. Taking once more the adjoint of this restricted map, we get a module map $\ddot{u}:  A^{**}\to M$. By the same argument as in \cite[Lemma 6.9]{P}, 
\vspace{.2cm}

$(vi)$\  $u \in D_{\mathfrak A}(A,M)$ if and only if $\ddot u \in D_{\mathfrak A}(A^{**},M)$ and  $\|u\|_{{\mathfrak A}-{\rm dec}}=\|\ddot u\|_{{\mathfrak A}-{\rm dec}}$.  
\vspace{.2cm}

\begin{pro}\label{61}
Let $ A, B $ and $ C $ be  $ C^* $-algebras. For each $ u\in  D_{\mathfrak A}(A,B) $ and $ x\in C\odot_{\mathfrak A} A$ we have 
$$ \|({\rm id}_C\odot_{\mathfrak A}u)(x)\|_{C\otimes_{{\mathfrak A}-{\rm max}}B}\leq\| u \|_{{\mathfrak A}-{\rm dec}} \|x \|_{C\otimes_{{\mathfrak A}-{\rm max}}A}. $$
Moreover, the map ${\rm id}_C\otimes_{\mathfrak A,\rm max }u:C\otimes_{{\mathfrak A}-{\rm max}}A\rightarrow C\otimes_{{\mathfrak A}-{\rm max}}B $ is decomposable and its norm satisfies
\begin{equation}\label{6.11}
 \|({\rm id}_C\otimes_{\mathfrak A,\rm max}u)\|_{D_{\mathfrak A}(C\otimes_{{\mathfrak A}-{\rm max}}A,C\otimes_{{\mathfrak A}-{\rm max}}B)}\leq\| u \|_{{\mathfrak A}-{\rm dec}}.
\end{equation}
\end{pro}
\begin{proof}
The proof is based on the argument in \cite[Proposition 6.11]{P} and we give the details for the sake of completeness.	
By (\ref{6.1}), $\|{\rm id}_C\odot_{\mathfrak A}u \|\leq\| {\rm id}_C\odot_{\mathfrak A}u\|_{{\mathfrak A}-{\rm dec}} $ and so we only need  to prove (\ref{6.11}).  Let $ u\in  D_{\mathfrak A}(A,B) $ and let
 $V=\begin{pmatrix}
	S_1& u\\
	u_*& S_2
\end{pmatrix}
\in CP(A, \mathbb M_2(B)).$ Then the map ${\rm id}_C\otimes_{\mathfrak A,\rm max }V:C\otimes_{{\mathfrak A}-{\rm max}}A\rightarrow C\otimes_{{\mathfrak A}-{\rm max}}\mathbb M_2(B)\simeq\mathbb M_2(C\otimes_{{\mathfrak A}-{\rm max}}B)  $ is c.p. and we may view it as 
 $$ {\rm id}_C\otimes_{\mathfrak A,\rm max }V=\begin{pmatrix}
	{\rm id}_C\otimes_{\mathfrak A,{\rm max}}S_1& {\rm id}_C\otimes_{\mathfrak A,{\rm max}}u\\
{\rm	id}_C\otimes_{\mathfrak A,{\rm max}}u_*&  {\rm id}_C\otimes_{\mathfrak A,{\rm max}}S_2.
\end{pmatrix}
$$ Since ${\rm	id}_C\otimes_{\mathfrak A,{\rm max}}u_*=(	{\rm id}_C\otimes_{\mathfrak A,{\rm max}}u)_*  $ and $ u\in  D_{\mathfrak A}(A,B) $, we obtain ${\rm id}_C\otimes_{\mathfrak A,{\rm max}}u\in D_{\mathfrak A}(C\otimes_{{\mathfrak A}-{\rm max}}A,C\otimes_{{\mathfrak A}-{\rm max}}B) $ and 
$$  \|{\rm id}_C\otimes_{\mathfrak A,{\rm max}}u\|_{{\mathfrak A}-{\rm dec} }\leq {\rm max} \{ \| {\rm id}_C\otimes_{\mathfrak A,{\rm max}}S_1 \|, \|{\rm id}_C\otimes_{\mathfrak A,{\rm max}}S_2\|\} \leq {\rm max} \{ \| S_1 \|, \|S_2\|\}.$$
Taking the infimum over all possible module maps $S_1,S_2  $ we obtain (\ref{6.11}).
\end{proof}

\begin{cor}
	Let $ u_i\in  D_{\mathfrak A}(A_i,B_i), i=1,2 $. Then $  u_1\odot_{\mathfrak A} u_2 $ extends to a $ \mathfrak A $-decomposable map $ u_1\otimes_{\mathfrak A,\rm max} u_2 $ such that 
	$$\|  u_1\otimes_{\mathfrak A,\rm max} u_2\|_{{\mathfrak A}-{\rm dec}}\leq\|u_1\|_{{\mathfrak A}-{\rm dec}}\|u_2\|_{{\mathfrak A}-{\rm dec}}  $$
\end{cor}
When the map $ u $ has finite rank then a stronger result holds and one can go $\rm min\rightarrow \rm max $.
\begin{pro}
	Let $ u\in  D_{\mathfrak A}(A,B) $ be a finite rank module map. Then for every  $ C^* $-algebra
	$ C $ and each $ x \in C\odot_{\mathfrak A}A $ we have
	$$ \|({\rm id}_C\odot_{\mathfrak A}u)(x)\|_{C\otimes_{{\mathfrak A}-{\rm max}}B}\leq\| u \|_{{\mathfrak A}-{\rm dec}} \|x \|_{C\otimes_{{\mathfrak A}-{\rm min}}A}. $$
\end{pro}
\begin{proof}
 The  $\rm min$ and $ \rm max $ norms are equivalent on $ C\odot_{\mathfrak A}F $ for any finite dimensional subspace $ F\subseteq B $. Since the rank of $ u $ is finite, the map 
 $$ {\rm id}_C\odot_{\mathfrak A}u:C\otimes_{{\mathfrak A}-{\rm min}}A\rightarrow C\otimes_{{\mathfrak A}-{\rm max}}B$$ is bounded. Also $$ \|  {\rm id}_C\odot_{\mathfrak A}u:C\otimes_{{\mathfrak A}-{\rm max}}A\rightarrow C\otimes_{{\mathfrak A}-{\rm max}}B  \|\leq \|u\|_{{\mathfrak A}-{\rm dec}},$$ by (\ref{6.11}). Let $ q:C\otimes_{\rm max}A \rightarrow C\otimes_{\rm min}A $ be the canonical surjection and let $ I=I_{C,A} $. Since $ \overline{I}^{\rm max}\subseteq \overline{I}^{\rm min} $, $ q $ induces a surjection $ \tilde{q}:C\otimes_{{\mathfrak A}-{\rm max}}A\rightarrow C\otimes_{{\mathfrak A}-{\rm min}}A.  $
 Taking the open unit ball onto the open unit ball, thus it follows that
 \begin{align*}
 	\|{\rm id}_C\odot_{\mathfrak A}u:C\otimes_{{\mathfrak A}-{\rm min}}A\rightarrow C\otimes_{{\mathfrak A}-{\rm max}}B  \|&= \| {\rm id}_C\odot_{\mathfrak A}u:C\otimes_{{\mathfrak A}-{\rm max}}A\rightarrow C\otimes_{{\mathfrak A}-{\rm max}}B \| \\&\leq \|u\|_{{\mathfrak A}-{\rm dec}}. 
 \end{align*} 
\end{proof}

Next, let $ \rho:\mathfrak A\rightarrow  B(\mathcal{H}) $ be a *-homomorphism. We consider a two-sided action of $ \mathfrak A $ on $ \mathcal{H} $ given by $ \alpha\cdot \xi=\xi \cdot \alpha=\rho(\alpha)\xi $ for each  $ \xi \in\mathcal{H}, \alpha \in \mathfrak A. $
This gives a compatible action of  $ \mathfrak A $ on $  B(\mathcal{H}) $ via $ (T\cdot\alpha)(\xi) =(\alpha\cdot T)(\xi)=T(\alpha\cdot \xi) $. Now if  a von Neumann algebra $ M\subseteq B(\mathcal{H}) $ is an $ \mathfrak A $-bimodule with compatible action, then its commutant $ M^\prime\subseteq  B(\mathcal{H})$ is a closed submodule. 

In what follows we further assume that $ A $ is a unital 
 $ C^* $-algebra and a commutative $ \mathfrak A $-bimodule, that is, $$ \alpha\cdot a=a\cdot \alpha \ \ (a\in A, \alpha\in \mathfrak A).$$
 We let $ E\subseteq A $ be a closed submodule and Suppose that $ u:E\rightarrow M $ is a bounded $ \mathfrak A $-linear map. Then $ \hat{u}:M^\prime\odot_{ \mathfrak A}E\rightarrow  B(\mathcal{H})$ defined by 
$\hat{u}(x^\prime\otimes x)=x^\prime u(x)  $ for $x^\prime\in M^\prime, x\in M $ is an $ \mathfrak A $-linear map. 
Finally, we take $ M^\prime\otimes^{{\rm max}}_{ \mathfrak A}E $ to be the closer of $ M^\prime\odot_{ \mathfrak A}E $ in $ M^\prime\otimes_{{\mathfrak A}-{\rm max}}A $.
\begin{lem}\label{6.20}
	With the above notations, 
	 $ \hat{u}:M^\prime\odot_{ \mathfrak A}E\rightarrow  B(\mathcal{H})$ extends to a c.b. map
	on $ M^\prime\otimes^{{\rm max}}_{ \mathfrak A}E $ with $ \|\hat{u}:  M^\prime\otimes^{{\rm max}}_{ \mathfrak A}E\rightarrow  B(\mathcal{H})\|_{c.b.}\leq1$
	if and only if there is $\tilde{u}\in  D_{\mathfrak A}(A,B)  $ with $\|\tilde{u}  \|_{{\mathfrak A}-{\rm dec}}\leq1  $ extending $ u $.
\end{lem}
\begin{proof}
Suppose that there is an extension $ \tilde{u}\in  D_{\mathfrak A}(A,B) $ with $\|\tilde{u}  \|_{{\mathfrak A}-{\rm dec}}\leq1.  $ By  (\ref{6.1}) and (\ref{6.11}) we have  
$$\|{\rm id}_{M^\prime}\otimes_{{\mathfrak A}-{\rm max}} \tilde{u}: M^\prime\otimes_{{\mathfrak A}-{\rm max}}A\rightarrow M^\prime\otimes_{{\mathfrak A}-{\rm max}}M\|_{c.b.}\leq 1. $$
Therefore,
$$\|{\rm id}_{M^\prime}\otimes_{{\mathfrak A}-{\rm max}} u:M^\prime\otimes^{{\rm max}}_{ \mathfrak A}E\rightarrow M^\prime\otimes_{{\mathfrak A}-{\rm max}}M\|_{c.b.}\leq 1.$$
Since the product map $ p:M^\prime\odot_{\mathfrak A}M\rightarrow  B(\mathcal{H}) $ defines a *-homomorphism on $  M^\prime\otimes_{{\mathfrak A}-{\rm max}}M $ into $  B(\mathcal{H}) $
 with  c.b. norm at most one and $\hat{u}=p({\rm id}_{M^\prime}\odot_{\mathfrak A}u) $, we get $  \|\hat{u}:  M^\prime\otimes^{{\rm max}}_{ \mathfrak A}E\rightarrow  B(\mathcal{H})\|_{c.b.}\leq 1 $.

Conversely, assume that $  \|\hat{u}:  M^\prime\otimes^{{\rm max}}_{ \mathfrak A}E\rightarrow  B(\mathcal{H})\|_{c.b.}\leq1 $. Let $\mathcal{K} $ be a  Hilbert space such that we have faithful representation 
$ M^\prime\otimes_{{\mathfrak A}-{\rm max}}A\subseteq  B(\mathcal{K}) $.
Let $ C= M^\prime\odot_{\mathfrak A}1\subseteq B(\mathcal{K})$ and let $ \pi:C\rightarrow B(\mathcal{H}) $ be the map coming from the natural identification $  M^\prime\odot_{\mathfrak A}1\simeq M^\prime, $ which in turn follows from the fact that  $ B(\mathcal{H}) $ is a commutative $ \mathfrak A $-bimodule. 
Put $ \mathcal{E}=M^\prime\otimes^{{\rm max}}_{ \mathfrak A}E $
and note that $ \mathcal{E} $ is a $ C $-bimodule and $\hat{u}  $  is a $ C $-bimodule by \cite[Lemma 6.19]{P} applied to $ w=\hat{u}. $ In particular, we obtain a u.c.p. $ \mathfrak A $-module map
$ W:\mathcal{E} \rightarrow \mathbb{M}_2(B(\mathcal{H})) $ which admits a  u.c.p. extension 
$ \tilde{W}:\mathbb{M}_2(M^\prime\otimes_{{\mathfrak A}-{\rm max}}A  )\rightarrow \mathbb{M}_2(B(\mathcal{H})) $.
Let $ T $ be the restriction of $ \tilde{W} $ to $ \mathbb{M}_2(1\odot_{\mathfrak A}A) $.
Since $ A $ is a commutative $ \mathfrak A $-bimodule, we may identify $ A $ with $1\odot_{\mathfrak A}A  $ and view $ T $ as a mapping from $ \mathbb{M}_2(A) $ to $ \mathbb{M}_2(B(\mathcal{H})) $ with $ T(1)=\tilde{W}(1)=1 $. We claim that the $ \mathfrak A$-module map $ T $ has the following special form 

\begin{equation}\label{6.27}
	 \qquad T(x)=\begin{pmatrix}
		T_{11}(x_{11})& T_{12}(x_{12})\\
		T_{21}(x_{21})& T_{22}(x_{22})
	\end{pmatrix}\ \ (x \in \mathbb M_2(A))
\end{equation}
with $T_{12}\arrowvert_{E}=u  $ and  that $ T(\mathbb{M}_2(A))\subseteq \mathbb{M}_2(M) $.
Assuming the claim, one could conclude the proof as follows.
Since $ T $ is   u.c.p.,  $ \max \{\|	T_{11}\|,\|T_{22}  \| \}\leq\| T\|=1 $,
the maps $ 	T_{11},T_{22} $ are c.p. and moreover the $ \mathfrak A$-module map $ R:a\mapsto T(\begin{pmatrix}
	a& a\\
a& a
\end{pmatrix}) $
is c.p. on $ A $ and self-adjoint. Let $ \tilde{u}=T_{12} $ and 
$ \tilde{u}_*(a)=\tilde{u}(a^*)^*, $ for  $ a\in A. $
Thus
$$ R(a)=\begin{pmatrix}
	T_{11}(a)& \tilde{u}(a) \\
	\tilde{u}_*(a)& T_{22}(a)
\end{pmatrix}.
$$
for each  $ a\in A. $ Therefore, by definition of the $ \mathfrak A$-dec norm, 
$$ \|\tilde{u}\|_{ \mathfrak A-\rm dec} \leq\max \{\|	T_{11}\|,\|T_{22}  \| \}=1. $$
It  only remains to prove the claim. To this end, first observe that by  definition
$ W $ is a *-homomorphism on the algebra of matrices $ \begin{pmatrix}
	c_1& 0 \\
	0& c_2
\end{pmatrix} $
with $ c_1, c_2\in C $. Let $ \mathcal{D} $ denote the set of all such matrices. By \cite[Corollary 5.3]{P}, the map $ \tilde{W} $ must satisfy 
\begin{equation}\label{6.28}
\tilde{W}(y_1xy_2)=W(y_1)\tilde{W}(x)W(y_2)
\end{equation}
for all $ y_1, y_2\in \mathcal{D}$ and $ x\in \mathbb{M}_2(B(\mathcal{H})) $.
Applying this with $ y_1, y_2 $ equal to either $  \begin{pmatrix}
	1& 0 \\
	0& 0
\end{pmatrix} $ or $  \begin{pmatrix}
0& 0 \\
0& 1
\end{pmatrix}, $
one concludes that  $ \tilde{W}(x)_{ij} $ depends only on
$ x_{ij} $ and the same is true for $ T=\tilde{W}\arrowvert_{\mathbb{M}_2(A)} $.
Thus we can write a priori $ T $ in the form (\ref{6.27}). Moreover, since $ \tilde{W} $ extends
$ W $, we know $ \tilde{W}_{12}$ extends $ W_{12}=w=\hat{u} $, and hence restricting  this to 
$ A\simeq 1\odot_{\mathfrak A}A $, on which $\hat{u}=u  $, we conclude that $ \tilde{u} $ extends $ u. $

Finally, it remains to check that all $ T_{ij} $s take their values in $ M $, equivalently, it suffices to check that all the terms $  T_{ij}( x_{ij})$ with $x_{ij}\in A$ commute with 
$ M^\prime $. Since $ 1\odot_{\mathfrak A}A $ commutes with $ M^\prime \odot_{\mathfrak A}1$, each 
$ x\in \mathbb{M}_2(A) $ commutes with any $ y\in \mathcal{D} $ of the form
$y= \begin{pmatrix}
	m^\prime\otimes 1& 0 \\
	0& m^\prime\otimes 1
\end{pmatrix} $
with $ m^\prime\in M^\prime $, and hence by (\ref{6.28}), $$ W(y)\tilde{W}(x)=\tilde{W}(yx)=\tilde{W}(xy)=\tilde{W}(x)W(y). $$
Equivalently, $ W(y)T(x)=T(x)W(y) $ and since $ W(y)=\begin{pmatrix}
	m^\prime & 0 \\
	0& m^\prime 
\end{pmatrix} $, this implies that all elements $ T_{ij}(x_{ij}) $ commute with any $ m^\prime\in M^\prime $ and hence take their values in $ M $. This completes the proof of the claim.
\end{proof}
\begin{lem}\label{7.3}
 For   $ \mathfrak A$-bimodules $ B, C, $ and each $ x\in C\odot_{\mathfrak A}B$ we have
 $$\| x\|_{C\otimes_{\mathfrak A-\rm max}B}=\| x\|_{C\otimes_{\mathfrak A-\rm max}B^{**}}.  $$
  \end{lem}
\begin{proof}
Let $ \pi:C\rightarrow B(\mathcal{H})$ and $ \sigma:B\rightarrow B(\mathcal{H}) $ be representations with commuting ranges and let $ \sigma^{**}:B^{**}\rightarrow \pi(C)^\prime $. By Theorem \ref{main1}, we have $$ \|\pi\times \sigma(x)\|_{B(\mathcal{H})}=\|\pi\times \sigma^{**}(x)\|_{B(\mathcal{H})}\leq\| x\|_{C\otimes_{\mathfrak A-\rm max}B^{**}},$$ whence $ \| x\|_{C\otimes_{\mathfrak A-\rm max}B}\leq\| x\|_{C\otimes_{\mathfrak A-\rm max}B^{**}} $. The converse is obvious.
\end{proof}
\begin{thm}
	Let $ E $ be a submodule of $ A $.
The  $ \mathfrak A$-module map $ u:E\rightarrow B $ is $\mathfrak{A}$-$(\rm max \to \rm max)$-tensorizing if and only if $ u $ admits a decomposable extension $ \tilde{u}:A\rightarrow B^{**} $ with $ \|\tilde{u}\|_{\mathfrak A-\rm dec} \leq1$ such that the following diagram commutes.
	
\begin{center}
	\begin{tikzpicture}
		\matrix [matrix of math nodes,row sep=1cm,column sep=1cm,minimum width=1cm]
		{
			|(A)| \displaystyle E  &   |(B)|  A    \\
			|(C)|  B    &   |(D)|  B^{**} \\
		};
		\draw[->]    (A)-- node [above] {$ {\rm id} $ }(B);
		\draw[->]   (A)--  node [left] { $ u $} (C);
		\draw[->]  (C)-- node [below]   { $i_B$}(D);
		\draw[->]  (B)-- node [right]  {$ \tilde{u} $}(D);
	\end{tikzpicture}
\end{center}

\end{thm}
\begin{proof}
Assume we have an extension $ \tilde{u} \in D_{\mathfrak A}(A,B^{**}) $ with $ \|\tilde{u}\|_{\mathfrak A-\rm dec} \leq1$. By Proposition \ref{61}, $\tilde{u}  $ and its restriction $ i_Bu $ are $\mathfrak{A}$-$(\rm max \to \rm max)$-tensorizing. By Lemma \ref{7.3}, the map $ u $ itself is  $\mathfrak{A}$-$(\rm max \to \rm max)$-tensorizing.

Conversely, assume that $ u $  is  $\mathfrak{A}$-$(\rm max \to \rm max)$-tensorizing.
We put $ C=M^\prime $ and apply Theorem \ref{6.20}. Let $ M=B^{**} $ be viewed as a von Neumann algebra embedded in $ B(\mathcal{H}), $ for some $ \mathcal{H}, $ such that $ B\subseteq B^{**}\subseteq B(\mathcal{H}) $. Let $ \hat{u}: M^\prime\otimes^{{\rm max}}_{ \mathfrak A}E\rightarrow  B(\mathcal{H})$ be as in Theorem  \ref{6.20}. Then we have 
$$ \|{\rm id}_{M^\prime}\odot_{\mathfrak{A}}u:M^\prime \otimes^{{\rm max}}_{ \mathfrak A}E\rightarrow M^\prime\otimes_{\mathfrak{A}-\rm max}B\|_{c.b.}\leq 1.     $$
Next, $\hat{u}  $ is the composition
of the map $ {\rm id}_{M^{'}}\odot_{\mathfrak{A}}u $ with  *-homomorphism 
$ \sigma:M^\prime \otimes_{\mathfrak{A}-\rm max}B\rightarrow B(\mathcal{H}) $ defined by $ \sigma(c\otimes b)=cb=bc $. Therefore, $$  \|\hat{u}:  M^\prime\otimes^{{\rm max}}_{ \mathfrak A}E\rightarrow  B(\mathcal{H})\|_{c.b.}\leq 1. $$ Now applying Theorem  \ref{6.20} to the map $ i_Bu:E\rightarrow M=B^{**} $ shows that there is an extension $ \tilde{u} \in D_{\mathfrak A}(A,B^{**}) $ with $ \|\tilde{u}\|_{\mathfrak A-\rm dec} \leq1$.

\end{proof}

\section*{Acknowledgment}

The author thanks to Professor Massoud Amini for his valuable comments and supports.

\section*{Compliance with Ethical Standards}

Funding: Not applicable.

Conflict of Interest: The authors declare that they have no known competing financial interests or personal relationships that could have appeared to influence the work reported in this paper.

Ethical Conduct: Not applicable.

Data Availability Statements: All input data for the simulations conducted in this study are available within the article.

%%%%%%%%%%%%%%%%%%%%%%%%%%%%%%%%%%%
%%%%%%%%%%%%%%%%%%%%%%%%%%%%%%%%%%%
%%%%%%%%%%%%%%%%%%%%%%%%%%%%%%%%%%%%

\end{document}